\documentclass{amsart}
\newtheorem{theorem}{Theorem}[section]
\newtheorem{corollary}[theorem]{Corollary}
\newtheorem{problem}[theorem]{Problem}
\newtheorem{definition}[theorem]{Definition}
\newtheorem{proposition}[theorem]{Proposition}
\begin{document}
\author{Gerardo Hern\'andez-del-Valle}
\title[On Schr\"odinger's equation, 3-dimensional bessel bridges]{On Schr\"odinger's equation,
 3-dimensional Bessel bridges, and passage time problems}
\address{Statistics Deparment, Columbia University\\ Mail Code
4403, New York, N.Y.} \email{gerardo@stat.columbia.edu}
\date{December 2007}
\begin{abstract}
We obtain explicit solutions for the density $\varphi_T$ of the first-time $T$ that a one-dimensional Brownian process  $B$ reaches the twice, continuously differentiable  moving boundary $f$ and such that  $f''(t)\geq 0$ for all $t\in \mathbb{R}^+$.

We do so by finding the expected value of some functionals of  a 3-dimensional Bessel bridge $\tilde{X}$ and exploiting its relationship with first-passage time problems as pointed out by Kardaras (2007). It turns out that this problem is related to Schr\"odinger's equation with time-dependent linear potential, see Feng (2001).
\end{abstract}

\thanks{Statistics Department, Columbia University, Mail Code
4403, New York, N.Y.} \maketitle

\section{Introduction}
%In this paper we use the existing relationship between a 3-dimensional Bessel bridge process [for a %general overview on Bessel process, the reader may consult for instance: Chapter 11 of Revuz and %Yor (2005)]  and first-passage time problems. The solution of one problem implies the solution of the %other, as has been pointed out by Kardaras (2007).

  The main aim of this paper is two-folded, given that $f\in\mathbb{C}^2[0,\infty)$, and such that $f''(t)\geq 0$, for all $t>0$: 
  \begin{enumerate}
\item  we compute
  \begin{equation}\label{gamma}
  \mathbb{E}\left[\exp\left\{-\int_0^sf''(u)\tilde{X}_udu \right\}\right],
  \end{equation}
where $\tilde{X}$ is a a 3-dimensional Bessel bridge starting at level $a>0$ [where $a=f(0)$] and reaching 0 at some fixed time $s$, and 
  \item as a corollary, we derive an explicit solution for the density $\varphi_f$ of the first time $T$ that a one-dimensional Brownian process $B$ reaches  the moving boundary $f$.
\end{enumerate}
  % of some functionals of a 3-dimensional Bessel bridge $\tilde{X}$.

The derviation of  $\varphi_f$ is an old and well studied  problem in probability and stochastic processes.
Yet, explicit solutions existed only in a limited number of special cases.

 The appeal of this problem stems not only from its applications (in finance, physics, biology, etc.) but also from its mathematical
 nature. As it turns out the problem may be studied through several different approaches:
 \begin{enumerate}
 \item  {\it Volterra integral equations\/}:  Several authors have derived Volterra integral equations in the study of the first-passage time problems [the reader may consult for instance Peskir (2002) for a detailed technical and historical account on the subject]. Yet, in the case in which the equation is not a convolution, or its kernel is not separable, numerical procedures are typically employed. Thus the results presented in this work provide explicit solutions to previously unknown integral equations. \\
 \item {\it Partial differential equations\/}: Groeneboom (1987), Salminen (1988), and Martin-L\"of (1998) tackled the problem of the quadratic boundary by solving a parabolic differential equation, derived from Girsanov's and Feynman-Kac theorems. In this paper not only do we generalize this result (an in fact provide an alternative proof for this special case), but also show how boundary problems are related to Schr\"odinger's equation with time-dependent linear potential.\\
\item {\it Bessel Bridge\/}: Kardaras (2007) notes that the problem of finding $\varphi_f$ is equivalent to the study of (\ref{gamma}). Several authors, for instance Revuz and Yor (2005),  have analyzed (\ref{gamma}), yet our results are new since we do not treat the squared process $\tilde{X}^2$.\\
\item {\it Montecarlo simulation\/}.
 \end{enumerate}

Roughly speaking, the solution to (\ref{gamma}), is based upon transforming a Cauchy problem (derived from Feynman-Kac's theorem), equations (\ref{cauchy1}) and (\ref{cauchy2}) into that of solving the Schr\"odinger equation (\ref{schro}) with time-dependent linear potential [see Feng (2001)]. And then proceed backwards to construct the corresponding Green's function.

In Section \ref{sec1}, after stating the problem we:  {\it (a)\/} point out the existing relationship  between a three-dimensional Bessel {\it bridge\/} process and  first-passage time problems. {\it (b)\/} We present our main results, Theorem \ref{thm} and Corollary \ref{coro}, which we will prove in the remainder of the paper.

In Section \ref{fey}, we state the problem in terms of the Feynman-Kac equation. We proceed in solving it by finding particular solutions with respect to the backward  $(t,a)$ and forward $(\tau,b)$ variables, in Sections \ref{sec2} and \ref{ff2} respectively. We conclude in Section \ref{gen} by using these results to construct the necessary Green's function $G$.

Both the backward and forward equations in Sections \ref{sec2} and \ref{ff2} are dealt with in a similar fashion: {\it (1)\/} Through algebraic transformations we obtain  Schr\"odinger's equation with time-dependent linear potential. {\it (2)\/} We solve the corresponding Schr\"odinger's equation, by reducing it into the heat equation.

%\section{Solution}

\section{The problem}\label{sec1}
\begin{problem}
The main motivation of this work is finding the density $\varphi_T$ of the first time $T$, that a one-dimensional, standard Brownian motion $B$ reaches the moving boundary $f$:
\begin{equation}\label{stop}
T:=\inf\left\{t\geq 0|B_t=f(t)\right\}
\end{equation}
where $f$ is twice, continuously differentiable, and  $f''(t)\geq 0$ for all $t>0$.
\end{problem}
 Alternatively, from Girsanov's theorem, the ``problem'' may be restated as:
\begin{proposition} Given that
\begin{equation}\label{level}
\varphi_a(t):=\frac{a}{\sqrt{2\pi t^3}}\exp\left\{-\frac{a^2}{2t}\right\}\qquad t\geq 0,\enskip a>0
\end{equation}
is the density of the first time that $B$ reaches the fixed level $a\in(0,\infty)$, the process $\tilde{X}$ is a 3-dimensional Bessel bridge, which has the following dynamics:
\begin{equation*}
d\tilde{X}_t=dW_t+\left(\frac{1}{\tilde{X}_t}-\frac{X_t}{s-t}\right)dt,\qquad \tilde{X}_0=a,\enskip t\in[0,s).
\end{equation*}
(Where $W$ is a Wiener process.)
Then, the distribution of $T$ equals
\begin{eqnarray*}
\mathbb{P}(T<t)&=&
\int_0^t\tilde{\mathbb{E}}\left[\exp\left\{-\int_0^sf''(u)\tilde{X}_udu\right\}\right]\\
&&\quad\times \exp\left\{-\frac{1}{2}\int_0^s(f'(u))^2du-f'(0)a\right\}\varphi_a(s)ds
\end{eqnarray*}
\end{proposition}
\begin{proof} From Girsanov's theorem, and the fact that the boundary $f$ is twice continuously differentiable, the following Radon-Nikodym derivative:
\begin{eqnarray*}
\frac{d\mathbb{P}}{d\tilde{\mathbb{P}}}(t)&:=&\exp\left\{-\int_0^tf'(s)d\tilde{B}_s-\frac{1}{2}\int_0^t(f'(s))^2ds\right\}\\
&:=&\exp\left\{-f'(t)\tilde{B}_t+\int_0^tf''(s)\tilde{B}_sds-\frac{1}{2}\int_0^t(f'(s))^2ds\right\},
\end{eqnarray*}
is indeed a {\it martingale\/} and induces the following relationship:
\begin{eqnarray*}
\mathbb{P}(T\leq
t)&:=&\tilde{\mathbb{E}}\left[\frac{d\mathbb{P}}{d\tilde{\mathbb{P}}}(t)\mathbb{I}_{(T\leq
t)}\right]\\
&=&\tilde{\mathbb{E}}\left[\frac{d\mathbb{P}}{d\tilde{\mathbb{P}}}(T)\mathbb{I}_{(T\leq
t)}\right]\\
&=&\int_0^t\tilde{\mathbb{E}}\left[\frac{d\mathbb{P}}{d\tilde{\mathbb{P}}}(T)\Big{|}T=s\right]\varphi_a(s)ds,
\end{eqnarray*}
where the second equality follows from the optional sampling theorem, and the third   from conditioning with respect to $T$ under $\tilde{P}$, and $\varphi_a$ is defined in (\ref{level}).

Next, Kardaras (2007) shows that calculating [see for instance Chapter 11 in Revuz and Yor (2005), for a detailed overview of 3-dimensional Bessel bridges] :
\begin{equation}\label{expected}
\tilde{\mathbb{E}}\left[\frac{d\mathbb{P}}{d\tilde{\mathbb{P}}}(T)\Big{|}T=s\right]
\end{equation}
is equivalent to finding the expected value of a functional of a 3-dimensional Bessel bridge $\tilde{X}$, which at time 0 starts at $a$ and at time $s$ equals 0. Indeed,
\begin{eqnarray}
&&\nonumber\tilde{\mathbb{E}}\left[\frac{d\mathbb{P}}{d\tilde{\mathbb{P}}}(T)\Big{|}T=s\right]\\
&&\nonumber\qquad=
\tilde{\mathbb{E}}\left[\exp\left\{-f'(T)a+\int_0^Tf''(u)\tilde{B}_udu-\frac{1}{2}\int_0^T(f'(u))^2du\right\}\Big{|}T=s\right]\\
&&\nonumber\qquad=
\tilde{\mathbb{E}}\left[\exp\left\{-f'(s)a+\int_0^sf''(u)(a-\tilde{X}_u)du-\frac{1}{2}\int_0^s(f'(u))^2du\right\}\right]\\
&&\nonumber\qquad =
\tilde{\mathbb{E}}\Bigg{[}\exp\Bigg{\{}-f'(s)a+(f'(s)-f'(0))a-\int_0^sf''(u)\tilde{X}_udu\\
&&\nonumber\qquad\quad-\frac{1}{2}\int_0^s(f'(u))^2du\Bigg{\} }\Bigg{]}\\
&&\label{problem}\qquad =
\tilde{\mathbb{E}}\left[\exp\left\{-\int_0^sf''(u)\tilde{X}_udu\right\}\right]\exp\left\{-\frac{1}{2}\int_0^s(f'(u))^2du-f'(0)a\right\}
\end{eqnarray}
 thus the process $a-\tilde{X}$  equals 0 at $t=0$, and at $s$ reaches level $a$ for the first time, as required.
\end{proof}

We conclude from equation (\ref{problem}) that our goal is to compute the following expected value:
\begin{equation}\label{ex}
\tilde{\mathbb{E}}\left[\exp\left\{-\int_0^sf''(u)\tilde{X}_udu\right\}\right].
\end{equation}

Equation (\ref{ex}) has the following solution:
\begin{theorem}\label{thm}
 Let the boundary $f\in\mathbb{C}^2[0,\infty)$ satisfy  $f''(t)\geq 0$ for all $t$. Then:
\begin{eqnarray*}
\tilde{\mathbb{E}}\left[\exp\left\{-\int_t^sf''(u)\tilde{X}_udu\right\}\right]=v(t,a)
\end{eqnarray*}
where
%satisfies both the backward and forward equations (\ref{kom1}) and (\ref{kom2}) in the backward $(s,a)%$ and the forward $(t,b)$ variables respectively.
%hence if
\begin{eqnarray}
\label{v} v(t,a)&=&\int_{0}^{\infty}G(t,a;s,b)db,\quad 0\leq t<s,\enskip a\in \mathbb{R}^+\backslash\{0\}
\end{eqnarray}
given that
\begin{eqnarray*}
G(t,a;\tau,b)=\frac{\varphi_b(s-\tau)}{\varphi_a(s-t)}H(t,a;\tau,b),\quad 0\leq t<\tau<s
\end{eqnarray*}
where
\begin{eqnarray*}
&&H(t,a;\tau,b)=\\
&&\enskip\frac{1}{\sqrt{2\pi(\tau-t)}}\exp\left\{\frac{1}{2}\int_t^\tau(f'(u))^2du-f'(\tau)b+f'(t)a\right\}\\
&&\enskip\times\Bigg{[}\exp\left\{-\frac{\left(b-a-\int_t^\tau f'(u)du\right)^2}{2(\tau-t)}\right\}
-\exp\left\{-\frac{\left(b+a-\int_t^\tau f'(u)du\right)^2}{2(\tau-t)}\right\}\Bigg{]}.
\end{eqnarray*}
Where $\varphi_a$ is defined in (\ref{level}).
\end{theorem}
Thus the density $\varphi_T$ of the stopping $T$  in equation (\ref{stop}) equals:
\begin{corollary}\label{coro} Let $B$ be a one-dimensional standard Brownian motion $B$ and let the stopping time $T$ be as in (\ref{stop}). Then, the density $\varphi_T$  equals:
\begin{eqnarray*}
\varphi_T(s)=v(0,a)\varphi_a(s)\exp\left\{-\frac{1}{2}\int_0^s(f'(u))^2du-f'(0)a\right\}.
\end{eqnarray*}
where $v$ and $\varphi_a$ are defined in (\ref{v}) and (\ref{level}) respectively.
\end{corollary}

%And has the following moments
%\begin{eqnarray*}
%\mathbb{E}^a\left(\frac{\tilde{X}_t}{\sqrt{s-t}}\right)^p=c_p\left(\frac{t}{s}\right)^{p/2}\exp\left\{-\frac{a^2}{2}\left(\frac{1}{t}-\frac{1}{s}\right)\right\}
%H\left(\frac{a^2}{2}\left(\frac{1}{t}-\frac{1}{s}\right);\frac{p+3}{2},\frac{3}{2}\right),\quad
%t\leq s.
%\end{eqnarray*}
%$c_p:=\mathbb{E}|Z|^{p+2}$.
%\begin{eqnarray*}
%c_1&=&\mathbb{E}|Z|^3\\
%&=&2 \sqrt{\frac{2}{\pi}}
%\end{eqnarray*}

\section{The Feyman-Kac equation}\label{fey}

As the reader may suspect from equation (\ref{ex}) and the Green's function appearing in equation (\ref{v}), in Theorem \ref{thm}, we will follow a P.D.E approach to solve our problem. To do so, we will first introduce Feynman-Kac's theorem and some necessary tools, and will conclude this section by introducing the corresponding Cauchy problem (see Problem \ref{bessie}). 

If we let
the operator $\mathcal{A}_t$ satisfy
\begin{eqnarray*}
(\mathcal{A}_tv)(t,a)&=&\frac{1}{2}\sigma^2(t,a)\frac{\partial^2v}{\partial
a^2}(t,a)+b(t,a)\frac{\partial v}{\partial a}(t,a)\\
&=&\frac{1}{2}\frac{\partial^2v}{\partial
a^2}(t,a)+\left(\frac{1}{a}-\frac{a}{s-t}\right)\frac{\partial
v}{\partial a}(t,a)\qquad 0\leq t<s,\enskip a\in\mathbb{R}^+\backslash\{0\},
\end{eqnarray*}
Then, we may introduce a simplified version of the celebrated Feynman-Kac theorem which relates {\it parabolic\/} partial differential equations with stochastic differential equations [for a detailed and general overview of this and related topics the reader may consult for instance: Chapter 5, in Karatzas and Shreve (1998)]: 

\begin{theorem}  [Feynman-Kac] Given the Cauchy problem
\begin{eqnarray*}
-\frac{\partial v}{\partial t}+kv=\mathcal{A}_tv+g;&&\enskip [0,T)\times \mathbb{R}\\
v(s,a)=\tilde{f}(a);&&\enskip a\in \mathbb{R}^+\backslash\{0\}
\end{eqnarray*}
where $k(t,x):[0,s]\times \mathbb{R}^+\to[0,\infty)$, and functions $f$ and $g$ satisfy the standard assumptions [see Theorem 5.7.6 in Karatzas and Shreve (1998)], and $v$ satisfies the polynomial growth condition. Then,
the function $v(t,a)$ admits the stochastic representation
\begin{eqnarray*}
v(t,a)&=&\tilde{\mathbb{E}}^{t,a}\Bigg{[}\tilde{f}(\tilde{X}_s)\exp\left\{-\int_t^sk(u,\tilde{X}_u)du\right\}\\
&&+\int_t^sg(v,\tilde{X}_v)\exp\left\{-\int_t^vk(u,\tilde{X}_u)du\right\}dv\Bigg{]}
\end{eqnarray*}
on $[0,s]\times \mathbb{R}^+$; in particular, such a solution is unique.
\end{theorem}
The following definition, will guide us on the construction of our Green's function $G$:
\begin{definition}\label{def2} [Definition 5.7.9 in Karatzas and Shreve (1998)]  A fundamental solution of the second-order partial differential equation
\begin{equation}\label{def1}
-\frac{\partial u}{\partial t}+ku=\mathcal{A}_t u
\end{equation}
is a nonnegative function $G(t,a;\tau,b)$ defined for $0\leq t\leq\tau\leq s$, $a\in \mathbb{R}^+$, $b\in\mathbb{R}^+$ with the property that for every $f\in C_0(\mathbb{R}^+)$, $\tau\in(0,s]$, the function
\begin{equation*}
u(t,a):=\int_0^\infty G(t,a;\tau,b)\tilde{f}(b)db;\quad 0\leq t<\tau, \enskip a\in \mathbb{R}^+
\end{equation*}
is bounded, of class $C^{1,2}$, satisfies (\ref{def1}), and
\begin{equation*}
\lim\limits_{t\uparrow\tau}u(t,a)=\tilde{f}(a),\quad a\in \mathbb{R}^+.
\end{equation*}
\end{definition}

For fixed $(\tau,b)\in(0,s]\times \mathbb{R}^+$, the function
$$\varphi(t,a):=G(t,a;\tau,b)$$
is of class $C^{1,2}$ and satisfies the backward Kolomogorov equation in the backward variables $(t,a)$. And, for fixed $(t,a)\in [0,s)\times \mathbb{R}^+$ the function
$$\psi(\tau,b):=G(t,a;\tau,b)$$
is of class $C^{1,2}$ and satisfies the forward Kolmogorov equation (also known as the Fokker-Planck equation).

\indent Hence, from equation (\ref{problem}) it follows that:
%$\tilde{X}_s=0$, $g\equiv 0$, thus
\begin{equation*}
\tilde{\mathbb{E}}^{0,a}\left[\exp\left\{-\int_0^sf''(u)\tilde{X}_udu\right\}\right]=v(0,a).
\end{equation*}
where
\begin{eqnarray*}
v(t,a)&:=&\mathbb{E}^{t,a}\Bigg{[}\tilde{f}(0)\exp\left\{-\int_t^sk(u,\tilde{X}_u)du\right\}\Bigg{]},
\end{eqnarray*}
for  $k(u,a):=f''(u)\cdot a$, and $v(T,a)=\tilde{f}(a)\equiv 1$, {\it i.e.\/},

\begin{problem}[Bessel's Cauchy problem]\label{bessie} Equation (\ref{ex}) satisfies the following Cauchy problem:
\begin{eqnarray}
\label{cauchy1}-\frac{\partial v}{\partial
t}(t,a)+f''(t)av(t,a)=\frac{1}{2}\frac{\partial^2v}{\partial
a^2}(t,a)&&\\
 \nonumber +\left(\frac{1}{a}-\frac{a}{s-t}\right)\frac{\partial
v}{\partial a}(t,a);&&\enskip \hbox{in }[0,s)\times \mathbb{R}^+\backslash \{0\}\\
\label{cauchy2}v(s,a)=1;&&\enskip a\in \mathbb{R}^+\backslash \{0\}.
\end{eqnarray}
\end{problem}

Definition \ref{def2} hints us towards the path we should follow. We should first find a general solution to the backward-Kolmogorov equation, in Section \ref{sec2}:
\begin{eqnarray*}
-\frac{\partial \varphi}{\partial
t}(t,a)+f''(t)a\varphi(t,a)=\frac{1}{2}\frac{\partial^2\varphi}{\partial
a^2}(t,a)&&\\
 \nonumber +\left(\frac{1}{a}-\frac{a}{s-t}\right)\frac{\partial
\varphi}{\partial a}(t,a);&&\enskip \hbox{in }[0,s)\times \mathbb{R}^+\backslash \{0\}\
\end{eqnarray*}
and to the forward-Kolmogorov equation, in Section \ref{ff2}:
\begin{eqnarray*}
\frac{\partial \psi(\tau,b)}{\partial \tau}(\tau,b)+f''_k(\tau)b\psi(\tau,b)=\frac{1}{2}\frac{\partial^2\psi}{\partial b^2}(\tau,b)\\
-\frac{\partial}{\partial b}\left[\left(\frac{1}{b}-\frac{b}{s-\tau}\right)\psi(\tau,b)\right]
;&&\enskip \hbox{in }(0,s]\times \mathbb{R}^+\backslash \{0\}.
\end{eqnarray*}
Which after proper change's of variable and algebraic transformations will equal respectively:
\begin{eqnarray*}
-\frac{\partial\varphi^3}{\partial t}(t,y)&=&\frac{1}{2}\frac{\partial^2\varphi^3}{\partial y^2}(t,y),\quad 0\leq t<s\\
\frac{\partial\psi^3}{\partial t}(\tau,z)&=&\frac{1}{2}\frac{\partial^2\psi^3}{\partial z^2}(\tau,z),\quad 0\leq \tau\leq s,\\
\end{eqnarray*}
and $y$ and $z$ are function of $a$ and $b$ respectively (and will be determined later).
Since both $\varphi$ and $\psi$ may be reduced to the Heat equation, and the variables $a,b$ are only defined on the half line. Then the Green's function $\tilde{G}$ which relates these two solutions will be of the form:
\begin{equation*}
\tilde{G}(t,y;\tau,z)\!=\!\frac{1}{\sqrt{2\pi(\tau-t}}\left(\exp\left\{-\frac{(z-y)^2}{2(\tau-t)}\right\}-\exp\left\{-\frac{(z+y)^2}{2(\tau-t)}\right\}\right)\enskip 0\leq t<\tau\leq s.
\end{equation*}
Thus, the construction of the particular solution $G$ will just be a backwards procedure.
\section{Schr\"odinger's backward equation}\label{sec2}

In this section we obtain a particular solution to  the backward equation (\ref{cauchy1})   by first transforming it into Schr\"odinger's backward-equation with time-dependent linear potential through algebraic transformations.

\begin{proposition} Equation (\ref{cauchy1}) satisfies the following relationship
\begin{equation}\label{partial}
\varphi(t,a)=\varphi^1(t,a)\left\{A(t)\exp\left[B(t)a^2\right]\right\},
\end{equation}
where
\begin{eqnarray}
\nonumber B(t)=\frac{1}{2(s-t)}\label{b}&&\\
\nonumber A(t)=c\cdot(s-t)^{3/2}\label{a}&&\\
af''(t)\varphi^1(t,a)-\frac{1}{2}\varphi^1_{aa}(t,a)-\frac{1}{a}\varphi^1_a(t,a)-\varphi^1_t(t,a)=0&&\label{partial2}
\end{eqnarray}
\end{proposition}
%\noindent{\bf Remark.} In particular, when $f''_k(s)=b$ (which is the case in which the boundary is quadratic), the partial differential equation (\ref{partial2}) is separable, hence a particular solution of $v$
%can be expressed as:
%\begin{eqnarray*}
%v(t,a,s)=\varphi_\lambda(a)\psi_\lambda(t)\left[A(t,s)\exp\left(B(t,s)a^2\right)\right].
%\end{eqnarray*}
%where
%\begin{eqnarray*}
%\frac{1}{2}\varphi''_{aa}+\frac{1}{a}\varphi'_a+\left(\lambda-ab\right)\varphi&=&0\\
%-\psi'_t+\lambda\psi&=&0.
%\end{eqnarray*}
%{\it i.e.,\/}
%\begin{eqnarray*}
%\varphi_\lambda(a)&=&\frac{1}{a}\left[c_1\hbox{Ai}\left(\frac{2^{1/3}(ab-\lambda)}{b^{2/3}}\right)+c_2\hbox{Bi}\left(\frac{2^{1/3}(ab-\lambda)}{b^{2/3}}\right)\right]\\
%\psi_\lambda(t)&=&\exp\left\{\lambda t\right\}.
%\end{eqnarray*}
%And thus, the {\it particular\/} solution can be constructed [see Martin-L\"of (1998)].\\[0.3cm]
\noindent{\it Proof\/}.
Let us first start by analyzing the transformation in equation (\ref{partial}):
\begin{eqnarray*}
\frac{\partial \varphi}{\partial a}(t,a)&=&\varphi^1_a(t,a)A(t)\exp\left[B(t)a^2\right]\\&&+2a \varphi^1(t,a) A(t)B(t)\exp\left[B(t)a^2\right]\\
\frac{\partial^2 \varphi}{\partial a^2}(t,a)&=&\varphi^1_{aa}(t,a)A(t)\exp\left[B(t)a^2\right]\\&&+2a\varphi^1_a(t,a) A(t)B(t)\exp\left[B(t)a^2\right]\\
&&+2\varphi^1(t,a) A(t)B(t)\exp\left[B(t)a^2\right]\\&&+2a\varphi^1_a(t,a)A(t)B(t)\exp\left[B(t)a^2\right]\\
&&+4a^2\varphi^1(t,a) A(t)B^2(t)\exp\left[B(t)a^2\right]\\
\frac{\partial \varphi}{\partial t}(t,a)&=&\varphi^1(t,a) A_t(t)\exp\left[B(t)a^2\right]\\&&+a^2\varphi^1(t,a) A(t)B_t(t)\exp\left[B(t)a^2\right]\\
&&+\varphi_t(t,a)A(t)\exp\left[B(t)a^2\right].
\end{eqnarray*}
Substituting these equations into (\ref{cauchy1}) we get
\begin{eqnarray*}
&&-\frac{A_t(t)}{A(t)}-a^2B_t(t)-\frac{\varphi^1_t}{\varphi^1}+a f''(t)=\frac{1}{2}\frac{\varphi^1_{aa}}{\varphi^1}+a\frac{\varphi^1_a}{\varphi}B(t)\\&&\qquad +B(t)
 +a\frac{\varphi^1_a}{\varphi^1}B(t)+2a^2B^2(t)+\left(\frac{1}{a}-\frac{a}{s-t}\right)\frac{\varphi^1_a}{\varphi^1}\\
&&\qquad+\left(\frac{1}{a}-\frac{a}{s-t}\right)2aB(t),
\end{eqnarray*}
equating to zero we have
\begin{eqnarray*}
&&-\frac{A_t(t)}{A(t)}-a^2B_t(t)-\frac{\varphi^1_t}{\varphi^1}+a f''(t)-\frac{1}{2}\frac{\varphi^1_{aa}}{\varphi^1}-2a\frac{\varphi^1_a}{\varphi^1}B(t)\\&&\qquad-B(t)
-2a^2B^2(t)-\frac{1}{a}\frac{\varphi^1_a}{\varphi^1}+\frac{a}{s-t}\frac{\varphi^1_a}{\varphi^1}\\
&&\qquad-2B(t)+2\frac{a^2}{s-t}B(t)=0
\end{eqnarray*}
finally, we factorize
\begin{eqnarray*}
&&-\left(\frac{A_t(t)}{A(t)}+3B(t)\right)-a^2\left(B_t(t)+2B^2(t)-2\frac{1}{s-t}B(t)\right)\\
&&-a\frac{\varphi^1_a}{\varphi^1}\left(2B(t)-\frac{1}{s-t}\right)+af''(t)-\frac{1}{2}
\frac{\varphi^1_{aa}}{\varphi^1}-\frac{1}{a}\frac{\varphi^1_a}{\varphi^1}-\frac{\varphi^1_t}{\varphi^1}=0.
\end{eqnarray*}
This equality holds for every $a>0$ if and only if
\begin{eqnarray*}
B(t)=\frac{1}{2(s-t)}\label{b}&&\\
A(t)=c\cdot(s-t)^{3/2}\label{a}&&\\
af''(t)\varphi^1(t,a)-\frac{1}{2}\varphi^1_{aa}(t,a)-\frac{1}{a}\varphi^1_a(t,a)-\varphi^1_t(t,a)=0&&\label{partial2}
\end{eqnarray*}
as claimed. $\Box$\\[0.3cm]
\subsection{The Schr\"odinger's backward equation} If we let $\varphi^2=a\cdot \varphi^1$, then
\begin{eqnarray*}
\varphi^1_t&=&\frac{1}{a}\varphi^2_t\\
\varphi^1_a&=&\frac{1}{a}\varphi^2_a-\frac{1}{a^2}\varphi^2\\
\varphi^1_{aa}&=&\frac{1}{a}\varphi^2_{aa}-\frac{1}{a^2}\varphi^2_a+\frac{2}{a^3}\varphi^2-\frac{1}{a^2}\varphi^2_a,
\end{eqnarray*}
after substituting in equation (\ref{partial2}):
\begin{eqnarray*}
af''(t)\varphi^2(t,a)-\frac{1}{2}\varphi^2_{aa}(t,a)-\varphi^2_t(t,a)=0.
\end{eqnarray*}
or, equivalently
\begin{equation}\label{schro}
-\frac{\partial \varphi^2}{\partial t}(t,a)=\frac{1}{2}\frac{\partial^2\varphi^2}{\partial a^2}(t,a)-af''(t)\varphi^2(t,a)
\end{equation}
which is  Schr\"odinger's backward-equation with time-dependent linear potential [see for instance, Feng (2001)].

%\noindent{\bf Observation.} From condition (\ref{cauchy2}), and equation (\ref{partial}) in Proposition xx  %we have that
%\begin{equation*}
%g(t,a)=\frac{a}{c(s-t)^{3/2}}\exp\left\{-\frac{a^2}{2(s-t)}\right\}v(t,a),
%\end{equation*}
%from which we may derive the boundary condition
%\begin{equation}\label{bound}
%g(s,a)=0.
%\end{equation}
%We will solve the Cauchy problem defined in equations (\ref{schro}) and (\ref{bound}) by first %constructing the necessary Green's function. Our function turns out to be a direct consequence of the %following Theorem which follows from translating Schr\"odinger's function into the heat equation.
%\\[0.3cm]
\begin{theorem}
Given the  backward Kolmogorov equation in (\ref{schro}):
\begin{equation}\label{kom1}
-\frac{\partial \varphi^2}{\partial t}(t,a)=\frac{1}{2}\frac{\partial^2 \varphi^2}{\partial a^2}(t,a)-af''(t)\varphi^2(t,a)
\end{equation}
%and its dual, the forward Kolmogorov equation
%\begin{equation}\label{kom2}
%\frac{\partial \tilde{g}}{\partial \tau}(\tau,b)=\frac{1}{2}\frac{\partial^2 \tilde{g}}{\partial b^2}(\tau,b)-%bf''(\tau)\tilde{g}(\tau,b)
%\end{equation}
we have the following {\it fundamental\/} solution,% correspondingly
\begin{eqnarray}
\label{sol1}\varphi^2(t,a)&=&\frac{\Im}{\sqrt{2\pi t}}\exp\left\{\frac{(a-f(t))^2}{2t}-\frac{1}{2}\int_0^t(f'(u))^2du+f'(t)a\right\},
%\label{sol2}\tilde{g}(\tau,b)&=&\frac{1}{\sqrt{2\pi \tau}}\exp\left\{-\frac{(b-f(\tau))^2}{2\tau}+\frac{1}{2}\int_0^\tau(f'(u))^2du-f'(\tau)b\right\}
\end{eqnarray}
where $\Im=\sqrt{-1}$.
\end{theorem}
\noindent{\it Proof of (\ref{sol1}).\/} First set $\varphi^2(t,a)=\lambda(t,a)e^{\beta(t)a}$, where $\beta$ will be determined later and compute
\begin{eqnarray*}
\frac{\partial \varphi^2}{\partial t}(t,a)&=&a\beta_t(t)\lambda(t,a)e^{\beta(t)a}+\frac{\partial\lambda}{\partial t}(t,a)e^{\beta(t)a}\\
\frac{\partial \varphi^2}{\partial a}(t,a)&=&\beta(t)\lambda(t,a)e^{\beta(t)a}+\frac{\partial\lambda}{\partial a}(t,a)e^{\beta(t)a}\\
\frac{\partial^2\varphi^2}{\partial a^2}(t,a)&=&\beta^2(t)\lambda(t,a)e^{\beta(t)a}+\beta(t)\frac{\partial\lambda}{\partial a}(t,a)e^{\beta(t)a}\\
&&+\frac{\partial^2\lambda}{\partial a^2}(t,a)e^{\beta(t)a}+\beta(t)\frac{\partial\lambda}{\partial a}e^{\beta(t)a}\\
&=&\beta^2(t)\lambda(s,a)e^{\beta(t)a}+2\beta(t)\frac{\partial\lambda}{\partial a}(t,a)e^{\beta(t)a}\\
&&+\frac{\partial^2\lambda}{\partial a^2}(t,a)e^{\beta(t)a}
\end{eqnarray*}
next, substitute into equation (\ref{kom1})
\begin{eqnarray*}
-\frac{\partial\lambda}{\partial t}(t,a)&=&\frac{1}{2}\frac{\partial^2\lambda}{\partial a^2}(t,a)+\frac{1}{2}\beta^2(t)\lambda(t,a)+\beta(t)\frac{\partial\lambda}{\partial a}(t,a)\\
&&+a(\beta_t(t)-f''(t))\lambda(t,a)
\end{eqnarray*}
Next, if we let $y=a-v(t)$ and set $\lambda(t,a)=u(t,y)$, then
\begin{eqnarray*}
-\frac{\partial u}{\partial t}(t,y)+v_t(t)\frac{\partial u}{\partial y}(t,y)&=&\frac{1}{2}\frac{\partial^2u}{\partial y^2}(t,y)+\frac{1}{2}\beta^2(t)u(t,y)+\beta(t)\frac{\partial u}{\partial y}(t,y)\\
&&+(y+v(t))(\beta_t(t)-f''(t))\lambda(t,y)
\end{eqnarray*}
to delete the term
\begin{eqnarray*}
\frac{\partial u}{\partial y}(t,y)
\end{eqnarray*}
set $v_t(t)=\beta(t)$, hence
\begin{eqnarray*}
-\frac{\partial u}{\partial t}(t,y)&=&\frac{1}{2}\frac{\partial^2u}{\partial y^2}(t,y)+\frac{1}{2}\beta^2(t)u(t,y)\\
&&+(y+v(t))(\beta_t(t)-f''(t))\lambda(t,y)
\end{eqnarray*}
furthermore, setting $\beta_t(t)=f''(t)$ we get
\begin{eqnarray*}
-\frac{\partial u}{\partial t}(t,y)&=&\frac{1}{2}\frac{\partial^2u}{\partial y^2}(t,y)+\frac{1}{2}\beta^2(t)u(t,y)
\end{eqnarray*}
which is a parabolic differential equation which now may be transformed into the heat equation by letting
$G(u)=1/2\beta^2(u)$ and
\begin{equation*}
\varphi^3(t,y)=u(t,y)e^{\int_0^tG(u)du}
\end{equation*}
it follows that
\begin{equation*}
-\frac{\partial \varphi^3}{\partial t}(t,y)=\frac{1}{2}\frac{\partial^2\varphi^3}{\partial y^2}(t,y)
\end{equation*}
which has general solution equal to:
\begin{equation*}
\varphi^3(t,y)=\frac{\Im}{\sqrt{2\pi t}}\exp\left\{\frac{y^2}{2t}\right\}
\end{equation*}
i.e.,
\begin{eqnarray*}
u(t,y)&=&=e^{-\int_0^tG(u)du}\varphi^3(t,y)\\
&=&e^{-\frac{1}{2}\int_0^t(f'(u))^2du}\varphi^3(t,y)\\
&=&\frac{\Im}{\sqrt{2\pi t}}\exp\left\{\frac{y^2}{2t}-\frac{1}{2}\int_0^t(f'(u))^2du\right\}
\end{eqnarray*}
and
\begin{eqnarray*}
\lambda(t,a)&=&\frac{\Im}{\sqrt{2\pi t}}\exp\left\{\frac{(a-v(t))^2}{2t}-\frac{1}{2}\int_0^t(f'(u))^2du\right\}\\
&=&\frac{\Im}{\sqrt{2\pi t}}\exp\left\{\frac{(a-f(t))^2}{2t}-\frac{1}{2}\int_0^t(f'(u))^2du\right\}.
\end{eqnarray*}
Hence
\begin{eqnarray*}
\varphi^2(t,a)&=&\frac{\Im}{\sqrt{2\pi t}}\exp\left\{\frac{(a-f(t))^2}{2t}-\frac{1}{2}\int_0^t(f'(u))^2du+\beta(t)a\right\}\\
&=&\frac{\Im}{\sqrt{2\pi t}}\exp\left\{\frac{(a-f(t))^2}{2t}-\frac{1}{2}\int_0^t(f'(u))^2du+f'(t)a\right\}
\end{eqnarray*}

%\noindent{\bf Proof\/.}

%\noindent{\bf Corollary.} The distribution is given by
%\begin{eqnarray*}
%\varphi_f(s)=\frac{1}{\sqrt{2\pi s}}\exp\left\{-\frac{(a-f(s))^2}{2s}-f'(s)a\right\}
%\end{eqnarray*}
%\noindent{\bf Corollary.} Since $g(0,a)=0$ from Duhamel's Principle it follows that
%\begin{equation*}
%g(s,a)=\int_0^s\tilde{g}(a,s-u)\delta(u)du
%\end{equation*}
%where
%\begin{eqnarray*}
%\delta(u)&=&\frac{\partial g}{\partial a}(u,0)\\
%&=&\frac{1}{\sqrt{2\pi u^3}}
%\end{eqnarray*}

\section{Schr\"odinger's Forward  equation}\label{ff2}
From equation (\ref{cauchy1}) we will now study its dual, the Fokker-Planck equation:
\begin{eqnarray*}
\frac{\partial \psi(\tau,b)}{\partial \tau}(\tau,b)+f''_k(\tau)b\psi(\tau,b)&=&\frac{1}{2}\frac{\partial^2\psi}{\partial b^2}(\tau,b)\\
&&-\frac{\partial}{\partial b}\left[\left(\frac{1}{b}-\frac{b}{s-\tau}\right)\psi(\tau,b)\right].%\\
%\frac{\partial \tilde{v}(\tau,b)}{\partial \tau}(\tau,b)+f''_k(\tau)b\tilde{v}(\tau,b)&=&\frac{1}{2}\frac{\partial^2\tilde{v}}{\partial b^2}(\tau,b).
%&&-\left(\frac{1}{b}-\frac{b}{s-\tau}\right)\frac{\partial\tilde{v}}{\partial b}(\tau,b)\\
%&&+\left(\frac{1}{b^2}+\frac{1}{s-\tau}\right)\tilde{v}(\tau,b)
\end{eqnarray*}
(The techniques used in this section are equivalent to those used when treating the backward-equation, with the obvious modifications.)
That is
\begin{eqnarray}
\label{f1}\frac{\partial \psi(\tau,b)}{\partial \tau}(\tau,b)&=&\frac{1}{2}\frac{\partial^2\psi}{\partial b^2}(\tau,b)\\
\nonumber&&-\left(\frac{1}{b}-\frac{b}{s-\tau}\right)\frac{\partial\psi}{\partial b}(\tau,b)\\
\nonumber&&+\left(\frac{1}{b^2}+\frac{1}{s-\tau}-f''_k(\tau)b\right)\psi(\tau,b)
\end{eqnarray}
\begin{proposition} Equation (\ref{f1}) satisfies the following relationship
%\noindent{\bf Proposition.}
\begin{equation}\label{partial2}
\psi(\tau,b)=\psi^1(\tau,b)\left\{A(\tau)\exp\left[B(\tau)b^2\right]\right\},
\end{equation}
where
\begin{eqnarray}
\nonumber B(\tau)=-\frac{1}{2(s-\tau)}\label{b}&&\\
\nonumber A(\tau)=c\cdot(s-\tau)^{-1/2}\label{a}&&\\
\nonumber \left(bf''(\tau)-\frac{1}{b^2}-\frac{1}{s-\tau}\right)\psi^1(\tau,b)-\frac{1}{2}\psi^1_{bb}(\tau,b)+\frac{1}{b}\psi^1_b(\tau,b)&&\\
+\psi^1_\tau(\tau,b)=0&&\label{p2}
\end{eqnarray}
\end{proposition}
\noindent{\it Proof\/}. Let us first start by analyzing the transformation in equation (\ref{p2}):
\begin{eqnarray*}
\frac{\partial \psi}{\partial b}(\tau,b)&=&\psi^1_b(\tau,b)A(\tau)\exp\left[B(\tau)b^2\right]\\&&+2b\psi^1(\tau,b)A(\tau)B(\tau)\exp\left[B(\tau)b^2\right]\\
\frac{\partial^2 \psi}{\partial b^2}(\tau,b)&=&\psi^1_{bb}(\tau,b)A(\tau)\exp\left[B(\tau)b^2\right]\\&&+2b\psi^1_b(\tau,b)A(\tau)B(\tau)\exp\left(B(\tau)b^2\right)\\
&&+2\psi^1(\tau,b)A(\tau)B(\tau)\exp\left[B(\tau)b^2\right]\\&&+2b\psi^1_bA(\tau)B(\tau)\exp\left[B(\tau)b^2\right]\\
&&+4b^2\psi^1(\tau,b)A(\tau)B^2(\tau)\exp\left[B(\tau)b^2\right]\\
\frac{\partial \psi}{\partial \tau}(\tau,b)&=&\psi^1(\tau,b)A_\tau(\tau)\exp\left[B(\tau)b^2\right]\\&&+b^2\psi^1(\tau,b)A(\tau)B_\tau(\tau)\exp\left[B(\tau)b^2\right]\\
&&+\psi^1_\tau(\tau,b) A(\tau)\exp\left[B(\tau)b^2\right].
\end{eqnarray*}
Substituting these equations into (\ref{f1}) we get
\begin{eqnarray*}
&&\frac{A_\tau(\tau)}{A(\tau)}+b^2B_\tau(\tau)+\frac{\psi^1_\tau}{\psi^1}+\left(b f''(\tau)-\frac{1}{b^2}-\frac{1}{s-\tau}\right)\\
&&\qquad=\frac{1}{2}\frac{\psi^1_{bb}}{\psi^1}+b\frac{\psi^1_b}{\psi^1}B(\tau)+B(\tau) +b\frac{\psi^1_b}{\psi^1}B(\tau)+2b^2B^2(\tau)\\&&\enskip\qquad -\left(\frac{1}{b}-\frac{b}{s-\tau}\right)\frac{\psi^1_b}{\psi^1}
-\left(\frac{1}{b}-\frac{b}{s-\tau}\right)2bB(\tau),
\end{eqnarray*}
equating to zero we have
\begin{eqnarray*}
&&\frac{A_\tau(\tau)}{A(\tau)}+b^2B_\tau(\tau)+\frac{\psi^1_\tau}{\psi^1}+\left(b f''(\tau)-\frac{1}{b^2}-\frac{1}{s-\tau}\right)\\
&&\qquad-\frac{1}{2}\frac{\psi_{bb}}{\psi^1}-2b\frac{\psi^1_b}{\psi^1}B(\tau)-B(\tau) -2b^2B^2(\tau)\\
&&\qquad +\left(\frac{1}{b}-\frac{b}{s-\tau}\right)\frac{\psi^1_b}{\psi^1}
+\left(\frac{1}{b}-\frac{b}{s-\tau}\right)2bB(\tau)=0,
\end{eqnarray*}
finally, we factorize
\begin{eqnarray*}
&&\left(\frac{A_\tau(\tau)}{A(\tau)}+B(\tau)\right)+b^2\left(B_\tau(\tau)-2B^2(\tau)-2\frac{1}{s-\tau}B(\tau)\right)\\
&&-b\frac{\psi^1_b}{\psi^1}\left(2B(\tau)+\frac{1}{s-\tau}\right)+\left(bf''(\tau)-\frac{1}{b^2}-\frac{1}{s-\tau}\right)\\
&&-\frac{1}{2}
\frac{\psi^1_{bb}}{\psi^1}+\frac{1}{b}\frac{\psi^1_b}{\psi^1}+\frac{\psi^1_\tau}{\psi^1}=0.
\end{eqnarray*}
This equality holds for every $b>0$ if and only if
\begin{eqnarray}
\nonumber B(\tau)=-\frac{1}{2(s-\tau)}&&\\
\nonumber A(\tau)=c\cdot(s-\tau)^{-1/2}\label{a}&&\\
\nonumber \left(bf''(\tau)-\frac{1}{b^2}-\frac{1}{s-\tau}\right)\psi^1(\tau,b)-\frac{1}{2}\psi^1_{bb}(\tau,b)+\frac{1}{b}\psi^1_b(\tau,b)&&\\
\nonumber +\psi^1_\tau(\tau,a)=0&&
\end{eqnarray}
as claimed. $\Box$
\subsection{The Schr\"odinger's forward equation}
We first start with the following transformation  $\psi^1=b\cdot \psi^2$, {\it i.e.\/}
\begin{eqnarray*}
\psi^1_b&=&\psi^2+b\cdot \psi^2_b\\
\psi^1_{bb}&=&2\psi^2_b+b\psi^2_{bb}\\
\psi^1_\tau&=&b\psi^2_\tau
\end{eqnarray*}
which after subtitution in (\ref{p2}) equals
\begin{eqnarray}
\nonumber b^2f''(\tau)\psi^2-\frac{1}{b}\psi^2-\frac{b}{s-\tau}\psi^2-\psi^2_b-\frac{1}{2}b\psi^2_{bb}+\frac{1}{b}\psi^2+\psi^2_b+b\psi^2_\tau&=&0\\
\nonumber b^2f''(\tau)\psi^2-\frac{b}{s-\tau}\psi^2-\frac{1}{2}b\psi^2_{bb}+b\psi^2_\tau&=&0\\
\label{2}\left(bf''(\tau)-\frac{1}{s-\tau}\right)\psi^2-\frac{1}{2}\psi^2_{bb}+\psi^2_\tau&=&0.
\end{eqnarray}
Next, setting $\psi^2$ equal to:
\begin{equation*}
 \psi^2=\frac{1}{s-\tau}\tilde{\psi}^2,
 \end{equation*}
 we have that
\begin{eqnarray*}
\psi^2_\tau&=&\frac{1}{(s-\tau)^2}\tilde{\psi}^2+\frac{1}{s-\tau}\tilde{\psi}^2_\tau\\
\psi^2_b&=&\frac{1}{s-\tau}\tilde{\psi}^2_b\\
\psi^2_{bb}&=&\frac{1}{s-\tau}\tilde{\psi}^2_{bb}.
\end{eqnarray*}
Thus subtituting into (\ref{p2})
\begin{eqnarray}
\nonumber\frac{bf''(\tau)}{s-\tau}\tilde{\psi}^2-\frac{1}{(s-\tau)^2}\tilde{\psi}^2-\frac{1}{2}\frac{1}{s-\tau}\tilde{\psi}^2_{bb}+\frac{1}{(s-\tau)^2}\tilde{\psi}^2+\frac{1}{s-\tau}\tilde{\psi}^2_\tau&=&0\\
\nonumber\frac{bf''(\tau)}{s-\tau}\tilde{\psi}^2-\frac{1}{2}\frac{1}{s-\tau}\tilde{\psi}^2_{bb}+\frac{1}{s-\tau}\tilde{\psi}^2_\tau&=&0\\
\label{al1}bf''(\tau)\tilde{\psi}^2-\frac{1}{2}\tilde{\psi}^2_{bb}+\tilde{\psi}^2_\tau&=&0
\end{eqnarray}
which is Schr\"odinger's forward equation.
\begin{theorem}
%\noindent{\bf Theorem.}
%Given the  backward Kolmogorov equation in (\ref{schro}):
%\begin{equation}\label{kom1}
%-\frac{\partial g}{\partial t}(t,a)=\frac{1}{2}\frac{\partial^2 g}{\partial a^2}(t,a)-af''(t)g(t,a)
%\end{equation}
Given the forward equation as in (\ref{al1})
\begin{equation}\label{kom2}
\frac{\partial \tilde{\psi}^2}{\partial \tau}(\tau,b)=\frac{1}{2}\frac{\partial^2 \tilde{\psi}^2}{\partial b^2}(\tau,b)-bf''(\tau)\tilde{\psi}^2(\tau,b)
\end{equation}
we have the following {\it fundamental\/} solution %s, correspondingly
\begin{eqnarray}
%\label{sol1}g(t,a)&=&\frac{\Im}{\sqrt{2\pi t}}\exp\left\{\frac{(a-f(t))^2}{2t}-\frac{1}{2}\int_0^t(f'(u))^2du+f'(t)a\right\}\\
\label{sol2}\tilde{\psi}^2(\tau,b)&=&\frac{1}{\sqrt{2\pi \tau}}\exp\left\{-\frac{(b-f(\tau))^2}{2\tau}+\frac{1}{2}\int_0^\tau(f'(u))^2du-f'(\tau)b\right\}
\end{eqnarray}
\end{theorem}
\noindent{\it Proof of (\ref{sol2})\/}.
First set $\tilde{\psi}^2(\tau,b)=\lambda(\tau,b)e^{-\beta(\tau)b}$, where $\beta$ will be determined later and compute
\begin{eqnarray*}
\frac{\partial \tilde{\psi}^2}{\partial \tau}(\tau,b)&=&-b\beta_\tau(\tau)\lambda(\tau,b)e^{-\beta(\tau)b}+\frac{\partial\lambda}{\partial \tau}(\tau,b)e^{-\beta(\tau)b}\\
\frac{\partial \tilde{\psi}^2}{\partial b}(\tau,b)&=&-\beta(\tau)\lambda(\tau,b)e^{-\beta(\tau)b}+\frac{\partial\lambda}{\partial b}(\tau,b)e^{-\beta(\tau)b}\\
\frac{\partial^2\tilde{\psi}^2}{\partial b^2}(\tau,b)&=&\beta^2(\tau)\lambda(\tau,b)e^{-\beta(\tau)b}-\beta(\tau)\frac{\partial\lambda}{\partial b}(\tau,b)e^{-\beta(\tau)b}\\
&&+\frac{\partial^2\lambda}{\partial a^2}(\tau,b)e^{-\beta(\tau)b}-\beta(\tau)\frac{\partial\lambda}{\partial b}e^{-\beta(\tau)b}\\
&=&\beta^2(\tau)\lambda(\tau,b)e^{-\beta(\tau)b}-2\beta(\tau)\frac{\partial\lambda}{\partial b}(\tau,b)e^{-\beta(\tau)b}\\
&&+\frac{\partial^2\lambda}{\partial b^2}(\tau,b)e^{-\beta(\tau)b}
\end{eqnarray*}
next, substitute into equation (\ref{kom2})
\begin{eqnarray*}
\frac{\partial\lambda}{\partial \tau}(\tau,b)&=&\frac{1}{2}\frac{\partial^2\lambda}{\partial b^2}(\tau,b)+\frac{1}{2}\beta^2(\tau)\lambda(\tau,b)-\beta(\tau)\frac{\partial\lambda}{\partial b}(\tau,b)\\
&&+b(\beta_\tau(\tau)-f''(\tau))\lambda(\tau,b)
\end{eqnarray*}
Next, if we let $z=b-v(\tau)$ and set $\lambda(\tau,b)=u(\tau,z)$, then
\begin{eqnarray*}
\frac{\partial u}{\partial \tau}(\tau,z)-v_\tau(\tau)\frac{\partial u}{\partial y}(\tau,z)&=&\frac{1}{2}\frac{\partial^2u}{\partial z^2}(\tau,z)+\frac{1}{2}\beta^2(\tau)u(\tau,z)-\beta(\tau)\frac{\partial u}{\partial z}(\tau,z)\\
&&+(z+v(\tau))(\beta_\tau(\tau)-f''(\tau))u(\tau,z)
\end{eqnarray*}
to delete the term
\begin{eqnarray*}
\frac{\partial u}{\partial z}(\tau,z)
\end{eqnarray*}
set $v_\tau(\tau)=\beta(\tau)$, hence
\begin{eqnarray*}
\frac{\partial u}{\partial \tau}(\tau,z)&=&\frac{1}{2}\frac{\partial^2u}{\partial z^2}(\tau,z)+\frac{1}{2}\beta^2(\tau)u(\tau,z)\\
&&+(z+v(\tau))(\beta_\tau(\tau)-f''(\tau))u(\tau,z)
\end{eqnarray*}
furthermore, setting $\beta_\tau(\tau)=f''(\tau)$ we get
\begin{eqnarray*}
\frac{\partial u}{\partial \tau}(\tau,z)&=&\frac{1}{2}\frac{\partial^2u}{\partial z^2}(\tau,z)+\frac{1}{2}\beta^2(\tau)u(\tau,z)
\end{eqnarray*}
which is a parabolic differential equation which now may be transformed into the heat equation by letting
$G(u)=1/2\beta^2(u)$ and
\begin{equation*}
\psi^3(\tau,z)=u(\tau,z)e^{-\int_0^\tau G(u)du}
\end{equation*}
it follows that
\begin{equation*}
\frac{\partial \psi^3}{\partial \tau}(\tau,z)=\frac{1}{2}\frac{\partial^2\psi^3}{\partial z^2}(\tau,z)
\end{equation*}
whose solution is
\begin{equation*}
\psi^3(\tau,z)=\frac{1}{\sqrt{2\pi \tau}}\exp\left\{-\frac{z^2}{2\tau}\right\}
\end{equation*}
{\it i.e.\/},
\begin{eqnarray*}
u(\tau,z)&=&=e^{\int_0^\tau G(u)du}\psi^3(\tau,z)\\
&=&e^{\frac{1}{2}\int_0^\tau(f'(u))^2du}\psi^3(\tau,z)\\
&=&\frac{1}{\sqrt{2\pi \tau}}\exp\left\{-\frac{y^2}{2\tau}+\frac{1}{2}\int_0^\tau(f'(u))^2du\right\}
\end{eqnarray*}
and
\begin{eqnarray*}
\lambda(\tau,b)&=&\frac{1}{\sqrt{2\pi \tau}}\exp\left\{-\frac{(b-v(\tau))^2}{2\tau}+\frac{1}{2}\int_0^\tau(f'(u))^2du\right\}\\
&=&\frac{1}{\sqrt{2\pi \tau}}\exp\left\{-\frac{(b-f(\tau))^2}{2\tau}+\frac{1}{2}\int_0^\tau(f'(u))^2du\right\}.
\end{eqnarray*}
Hence
\begin{eqnarray*}
\tilde{\psi}^2(\tau,b)&=&\frac{1}{\sqrt{2\pi \tau}}\exp\left\{-\frac{(b-f(\tau))^2}{2\tau}+\frac{1}{2}\int_0^\tau(f'(u))^2du-\beta(\tau)b\right\}\\
&=&\frac{1}{\sqrt{2\pi \tau}}\exp\left\{-\frac{(b-f(\tau))^2}{2\tau}+\frac{1}{2}\int_0^\tau(f'(u))^2du-f'(\tau)b\right\}
\end{eqnarray*}
\section{General solution}\label{gen}
From the backward (\ref{sol1}) and forward (\ref{sol2}) solutions of Schr\"odinger's equation, and the fact that $a,b\in \mathbb{R}^+\backslash\{0\}$ we have
\begin{proposition} Schr\"odinger's equation, see equations (\ref{kom1}) and (\ref{kom2}), with time-dependent linear potential and such that $a,b\in\mathbb{R}^+\backslash\{0\}$ has the following Green's function
\begin{eqnarray}
\label{g1}&&H(t,a;\tau,b)=\\
\nonumber&&\enskip\frac{1}{\sqrt{2\pi(\tau-t)}}\exp\left\{\frac{1}{2}\int_t^\tau(f'(u))^2du-f'(\tau)b+f'(t)a\right\}\\
\nonumber&&\times \enskip\Bigg{[}\exp\left\{-\frac{\left(b-a-\int_t^\tau f'(u)du\right)^2}{2(\tau-t)}\right\}
-\exp\left\{-\frac{\left(b+a-\int_t^\tau f'(u)du\right)^2}{2(\tau-t)}\right\}\Bigg{]},
\end{eqnarray}
where $0\leq t<\tau\leq s$.
\end{proposition}
\noindent{\it Proof\/}. Recall that the backward and forward equations may be reduced correspondingly to:
\begin{eqnarray*}
-\frac{\partial\varphi^3}{\partial t}(t,y)&=&\frac{1}{2}\frac{\partial^2\varphi^3}{\partial y^2}(t,y)\\
\frac{\partial\psi^3}{\partial \tau}(\tau,z)&=&\frac{1}{2}\frac{\partial^2\psi^3}{\partial z^2}(\tau,z),
\end{eqnarray*}
where $y=a-f(t)$, and $z=b-f(\tau)$, for $0\leq t<\tau<s$, and $a,b\in\mathbb{R}^+\backslash\{0\}$. Thus, the Green's function  $\tilde{G}$ that links these two solutions should be of the form: 
\begin{eqnarray*}
\tilde{G}(t,y;\tau,z)&=&\frac{1}{\sqrt{2\pi(\tau-t)}}\left(\exp\left\{-\frac{(z-y)^2}{2(\tau-t)}\right\}-\exp\left\{-\frac{(z+y)^2}{2(\tau-t)}\right\}\right).
\end{eqnarray*}
Finally, equation (\ref{g1}) follows from equations (\ref{kom1}) and (\ref{kom2}). 
\\[0.3cm]
Which alternatively leads to
%\begin{eqnarray}
%\label{green}H(s,a;t,b)&=&\frac{1}{\sqrt{2\pi(s-t)}}\exp\left\{-\frac{\left[b-a-\int_t^sf'(u)du\right]^2}{2(s-t)}%%\right\}\\
%&&\times\exp\left\{\frac{1}{2}\int_t^s(f'(u))^2du-f'(s)b+f'(t)a\right\}\nonumber\\
%&&\times\frac{b}{a}\exp\left\{-\frac{(b-a)^2}{2(s-t)}\right\}
%\end{eqnarray}
%or possibly\\
%\noindent{\bf Corollary} From equations (\ref{sol1}) and (\ref{sol2}) it follows that:
%\begin{eqnarray}
%\label{green}H(s,a;t,b)&=&\frac{1}{\sqrt{2\pi(s-t)}}\exp\left\{-\frac{\left[b-a-\int_t^sf'(u)du\right]^2}{2(s-t)}%%\right\}\\
%&&\times\exp\left\{\frac{1}{2}\int_t^s(f'(u))^2du-f'(s)b+f'(t)a\right\}\nonumber\\
%&&\times\frac{b}{a}
%\end{eqnarray}
\begin{proposition}\label{prop1} Our initial equation (\ref{cauchy1}) has the following Green's function:
\begin{eqnarray*}
G(t,a;\tau,b)=\frac{\varphi_b(s-\tau)}{\varphi_a(s-t)}H(t,a;\tau,b)
\end{eqnarray*}
\end{proposition}
\noindent{\it Proof\/}.  It follows from equations (\ref{partial}), (\ref{partial2}), and (\ref{g1}) and verification of Definition \ref{def2}.

\noindent{\bf Proof of Theorem \ref{thm}.} It follows from Proposition \ref{prop1}.
%\section{Concluding Remarks}

%In this paper we have derived an explicit equation for the density $\varphi_f$ of $T$ in the case in which $f>0$ and $f''>0$ yet we believe that this result can be extended to more general boundaries.
% (\ref{kom1}) and (\ref{kom2}) in the backward $(s,a)$ and the forward $(t,b)$ variables respectively.
%hence if

%i.e.
%\begin{eqnarray}
%\label{v} v(t,a)&=&\int_t^s\left(\int_{0}^{\infty}G(\tau,a;t,b)db\right)d\tau
%\end{eqnarray}
%\end{proposition}
%and the density $\varphi_T$ of the stopping $T$ defined in equation (\ref{stop}) equals
%\begin{eqnarray*}
%\varphi_T(s)=v(0,a)\varphi_a(s)\exp\left\{-\frac{1}{2}\int_0^s(f'(u))^2du-f'(0)a\right\}.
%\end{eqnarray*}
%Where $\varphi_a$ is defined in equation (\ref{level}).


\begin{thebibliography}{xx}
\bibitem{uno} Feng, M. (2001). Complete solution of the Schr\"odinger equation for the time-dependent linear potential, {\it Phys\/}. {\it Rev\/}. A {\bf 64}, 034101.
\bibitem{groe} Groeneboom, P. (1987). Brownian motion with a parabolic drift and Airy functions. {\it Prob. Theory Related Fields\/} {\bf 79}.
\bibitem{kar} Karatzas, I. and S. Shreve. (1991). Brownian Motion and Stochastic Calculus, Springer-Verlag, New York.
\bibitem{dos} Kardaras, K (2007). On the density of first passage times for difussions. {\it on preparation\/}.
\bibitem{martin} Martin-L\"of, A. (1998). The final size of a nearly critical epidemic, and the first passage time of a Wiener process to a parabolic barrier. {\it J. Appl. Prob.\/} {\bf 35}.
\bibitem{peskir} Peskir, G. (2001). On integral equations arising in the first-passage problem for Brownian motion, {\it J. Integral Equations Appl.\/}, {\bf 14\/}.
\bibitem{tres} Revuz, D., and M. Yor. (2005). Continuous martingales and Brownian motion, Springer-Verlag, New York.
\bibitem{salminen} Salminen, P. (1988). On the hitting time and last exit time for a Brownian motion to/from a moving boundary, {\it Advances in Applied Probability\/}  {\bf 20}.
\end{thebibliography}
\end{document}